\documentclass[leqno,11pt]{article}

\usepackage{amsmath, amssymb}
\usepackage{amsthm} 
\usepackage{amssymb}

\usepackage{mathrsfs}
\usepackage{amssymb, url, color, graphicx, amscd, mathrsfs}%pb-diagram} url,
\usepackage[colorlinks=true, bookmarks=true, pdfstartview=FitH, pagebackref=true, linktocpage=true, linkcolor = magenta, citecolor = blue]{hyperref}
 
\usepackage{graphicx}
\usepackage{times}
\usepackage{mathtools} %%% coloneqq
\usepackage{enumerate} 

\theoremstyle{plain} %text of this environment is typesetted in italics
\newtheorem{theorem}{\indent\sc Theorem}[section]
\newtheorem{lemma}[theorem]{\indent\sc Lemma}
\newtheorem{corollary}[theorem]{\indent\sc Corollary}
\newtheorem{proposition}[theorem]{\indent\sc Proposition}

\theoremstyle{definition} %text of this environment is typesetted in roman letters
\newtheorem{definition}[theorem]{\indent\sc Definition}
\newtheorem{remark}[theorem]{\indent\sc Remark}

\numberwithin{equation}{section}
\def\Rc{{\mathrm {Rc}}}
\def\scal{\mathrm{scal}}

\DeclareMathOperator{\Hess}{Hess}

\DeclareMathOperator{\I}{Isom}

\DeclareMathOperator{\R}{\mathbb{R}}
\DeclareMathOperator{\RP}{\mathbb{R}\mathrm{P}}

\DeclareMathOperator{\reg}{\mathrm{reg}}

\def\ddt{\frac{d}{dt}}

\makeatletter
\def\address#1#2{\begingroup
\noindent\parbox[t]{7.8cm}{%
\small{\scshape\ignorespaces#1}\par\vskip1ex
\noindent\small{\itshape E-mail address}%
\/: #2\par\vskip4ex}\hfill%
\endgroup}%
\makeatother
\title{K\"{a}hler-Ricci Solitons With Almost Maximal Symmetry} %title of the paper
\author{\textsc{Ha Tuan Dung, Catherine Searle, and Hung Tran}
\textsc{} %names of authors
}
\date{} %leave empty
\begin{document}

	\maketitle
	
	%%%%%%%%%%%%%%% footnote %%%%%%%%%%%%%%%%
	\footnote{ %2010 MSC numbers
	}
	\footnote{ %key words and phrases
		
	}
	
	%%%%%%%%%%%%%%%%%%%%%%%%%%%%%%%%%%%%%%%%%
	
	\begin{abstract}

This paper studies a non-trivial  gradient K\"{a}hler-Ricci soliton, of complex dimension $n$, with an isometry group of dimension at least $n^2-1$. We show that the isometry group acts by cohomogeneity one and, consequently, admits a special ansatz involving a Sasakian model. In complex dimension two, we can actually say more: namely that every such soliton  
has maximal symmetry; 
that is, the isometry group is exactly of dimension $2^2$. In addition, we prove that, if the isometry group acts by 
cohomogeneity one on a non-trivial gradient Ricci soliton (not necessarily K\"{a}hler), the potential function is invariant by the action. 
	\end{abstract}
	\section{Introduction}
	
	\quad\quad 
	In this paper we study gradient K\"{a}hler-Ricci solitons (GKRS) with almost maximal isometry groups. Recall that a complete connected Riemannian manifold $(M, g)$ is called a gradient Ricci soliton (GRS) if there exists a smooth potential function $f$ and a constant $\lambda$ such that
	\begin{equation}
		\label{grs}
		\Rc+\Hess f=\lambda g,
	\end{equation}
	where $\Rc$ denotes the Ricci curvature of $(M, g)$, and $\Hess (\cdot)$ is the Hessian of a function. A GRS is said to be {\em shrinking, steady}, or {\em expanding} if $\lambda>0, \lambda=0$, or $\lambda<0$, respectively. A GRS 
	can be viewed as a self-similar solution to the Ricci flow and, therefore, plays a fundamental role in the analysis of singularities. In particular, it follows from Enders, M\"uller, and Topping \cite{EMT} that a blow-up limit at a Type-I singularity always converges to a non-trivial GRS. As a consequence, the classification of GRS constitutes a central theme in the theory of Ricci flows.
	
	An Einstein manifold provides the simplest example of a GRS, corresponding to the case where  $\Hess f \equiv 0$ and $\lambda$ is equal to the Einstein constant. Another simple model is the Gaussian shrinking soliton $\left(\mathbb{R}^n, g_{\mathbb{R}^m}, \lambda \frac{|x|^2}{2}\right)$. A gradient Ricci soliton is called {\em rigid} if it is isometric to a quotient of a Riemannian product $\mathbb{N}^{m-k} \times \mathbb{R}^k$, where the potential function satisfies $f=\frac{|x|^2}{2}$ on the Euclidean factor. It was shown in Petersen and Wylie \cite{Pete2} that, whenever the metric of a GRS is reducible, the soliton structure decomposes accordingly. Thus, in this paper, a soliton is said to be {\em non-trivial}, or {\em non-rigid}, if it is locally irreducible and non-Einstein. 	

A large class of  
non-trivial examples of GRS arises in the K\"{a}hler setting, and the subject has received considerable interest: see, for instance, Bamler, Ciffarelli, and Conlon   \cite{BCC},  Cao \cite{Ca1}, Chen and Zhu \cite{Chen}, Ciffarelli, Conlon, and DeRuelle \cite{Cif}, Conlon, DeRuelle, and Sun \cite{Co}, Feldman, Ilmanen, and Knopf \cite{Fel}, Munteanu and Wang \cite{MW1},  Tian and Zhu \cite{Tian}, Tran \cite{Hung0, Hung1, Hung2}, and Wang and Zhu \cite{WZhu}. In complex dimension two, substantial progress has been made with a full classification of shrinking GKRS.  
We refer the reader to \cite{BCC, Cif, Co}, Li and Wang \cite{YL}, and the references therein. 

Let $\left(M, g, J\right)$ be an almost Hermitian manifold, that is, a $2n$-dimensional Riemannian manifold equipped with an almost complex structure, $J: T M \rightarrow T M$, compatible with the metric and satisfying $J^2=-\mathrm{Id}.$ The associated K\"{a}hler form $\omega$ is defined by
$
\omega(X, Y)=g(X, J Y)
$
for all tangent vector fields $X$ and $Y$. The structure is said to be {\em almost K\"{a}hler} if $\omega$ is closed, and {\em K\"{a}hler} if, in addition, $J$ is integrable. Throughout this paper, we adopt the convention that $n$ denotes the complex dimension while $m$ the real one. A GKRS $(M, g, J, f)$ is a manifold which is simultaneously K\"{a}hler and a GRS.

Isometric actions on Riemannian manifolds play a central role in differential geometry,  
giving us insight into 
how curvature and topological invariants are related to and constrain one another. 
A fundamental result is the 
theorem of Myers and Steenrod \cite{MyS} asserting that the full isometry group of any Riemannian manifold is a finite-dimensional Lie group. Moreover, it establishes a crucial connection between Riemannian geometry and Lie theory, allowing us to see that the symmetries of a Riemannian manifold are governed by both smooth and algebraic structures.

Using the theory of $ G$-structures developed in \cite{MyS}, Kobayashi \cite{Ko3} later provided an alternative approach. In particular, he showed that the isometry group $\I(M)$ of a connected Riemannian manifold $M$ of real dimension $m$ embeds naturally as a closed submanifold of the orthonormal frame bundle $O(M).$ As a consequence, one obtains the classical dimension bound $\dim \I(M) \le \frac{1}{2}m(m+1),$ with equality if and only if the manifold $M$ is either a simply connected space form or, if $\pi_1(M)$ is non-trivial, the real projective space. For an almost Hermitian manifold of complex dimension $n$, Tanno  \cite{Ta1} showed that the automorphism group has maximal dimension equal to $n^2+n$. This upper bound is necessarily smaller 
since the automorphism group preserves an additional tensor field. Thus, one expects that imposing the soliton structure on a manifold will lead to tightened bounds as confirmed by Tran \cite{DT} and Dung-Tran \cite{Hung1} for the real and K\"{a}hler cases, respectively. In particular, \cite{Hung1} proves that the isometry group of a GKRS has dimension at most $n^{2}$ and characterizes the equality case.  

  A natural next step is to understand the case in which the size of the symmetry group of a GKRS,  as measured by the dimension of its isometry group, is {\em almost maximal}. More precisely, we consider GKRS $(M,g,J,f,\lambda)$ of complex dimension $n$ with an almost maximal isometry group of dimension at least $n^2-1$. Our first result is the following structure theorem.

 \begin{theorem}\label{m1.1} Let $(M, g, J, f, \lambda)$ be a non-trivial GKRS of complex dimension $n\geq 2$. 
 Suppose that $\dim(\I(M))\geq n^2-1$ then  $\I(M)$ acts by cohomogeneity one and preserves $f$. 
 Moreover, the structure of the action is constrained as follows:
 \begin{enumerate}

 \item[\rm 1.] Each principal orbit is a homogeneous Sasakian manifold and a deformation of a fixed Sasakian structure $\left(P, \eta, \zeta, \Phi, g_P\right)$, which
is the total space of a Riemannian submersion with circle fibers and base a homogeneous K\"{a}hler Einstein manifold $\left(N, g_N, J_N\right)$; and
 \item[\rm 2.] There is at least one singular orbit corresponding to either a totally geodesic submanifold with a K\"{a}hler structure or a point. \end{enumerate}
		\end{theorem}
		
		\begin{remark}
			In fact, there is a more precise statement characterizing the metric on $M$ via the metric on $P$ in Theorem \ref{m1.1v2}. Indeed, it is reminiscent of classical Calabi-type ans\"atze for constructing K\"{a}hler metrics.  
			The structure and explicit form of the metric 
			result from the assumption on the dimension of the isometry group, reflecting a strong rigidity phenomenon at this level of symmetry.
		\end{remark}	
		
				\begin{remark}
		 It is classical that, on each regular level set of the potential function, $f$, the induced metric and almost complex structure give rise to an almost contact metric structure. Our observation is that, due to the almost maximal condition, it is a deformed homogeneous Sasakian manifold. A deformation re-scales the Reeb vector field and the transverse metric by different fixed factors on the same underlying manifold; see Definition \ref{Sdeform}. Furthermore, when a singular orbit corresponds to a point, then $\left(P, \eta, \zeta, \Phi, g_P\right)$ is the standard Sasakian sphere and $\left(N, g_N, J_N\right)$ is, up to homothety, isometric to a standard complex projective space.
		\end{remark}

The proof of Theorem \ref{m1.1} relies on two main ingredients: the fact that the isometric $G$-action turns out to be of cohomogeneity one and a precise characterization of the shape operator $S$ of a principal orbit. 
Of crucial importance is the following result, showing that 
the potential function $f$ is invariant by the cohomogeneity one action.  
	\begin{theorem}\label{t1.1}
			Let $\left(M, g, f\right)$ be a non-trivial GRS of real dimension $m \geq 3$. Suppose that  $\I(M)$ acts 
			on  $(M, g)$ by cohomogeneity one. Then the potential function, $f$, is invariant under the action.
	\end{theorem} 
	\begin{remark}
	In case $\lambda \neq 0$, Theorem \ref{t1.1} is  an immediate consequence of Corollary 2.2 of \cite{Pete2}. Our contribution is for the steady case where the argument in \cite{Pete2} cannot be  applied directly. Thus, this result is of independent interest and it is applicable to not necessarily K\"{a}hler GRS. 
	\end{remark}
	
In complex dimension two, we observe strong rigidity. 
\begin{theorem}\label{t3.1}
	Let $(M,g,J,f,\lambda)$ be a non-trivial GKRS of complex dimension two and assume that  $\dim(\I(M))\geq 3$. Then, $M$ is constructed from a homogeneous Sasakian manifold of constant holomorphic sectional curvature and it is of maximal symmetry, that is, $\dim(\I(M))=4$. 
\end{theorem}
\begin{remark}
	The maximal symmetry of a GKRS is analyzed in \cite{Hung1}. The manifold is foliated by equidistant connected level sets of $f$. The regular leaves of the foliation coincide with the principal orbits which, with the induced almost contact metric structure, are deformations of a Sasakian manifold with maximal symmetry. 
	By work of Tanno \cite{Ta}, the principal orbit must be a quotient of a simply-connected Sasakian space form by a cyclic group along trajectories of the Reeb vector fields.

	Theorem \ref{t3.1} tells us that there is a gap for the isometry groups of $2$-dimensional GKRS: 
	either $\dim(\I(M))=4$ or $\dim(\I(M))\leq 2$.
	In particular, one concludes that $SU(2)$ cannot be the full isometry group of a GKRS surface; without the K\"{a}hler requirement, there is construction of expanding GRS with $SU(2)$ isometry group \cite{Don}. 
\end{remark}

A consequence of Theorem \ref{t3.1} is the following rigidity result in the non-expanding case. 

\begin{corollary}\label{c1.8}
	Let $(M, g, J, f, \lambda)$ be an irreducible, non-trivial, complete, connected, GKRS of complex dimension two. If $\lambda \geq 0$ and $\dim(\I(M))\geq 3$, then it is of maximal symmetry, that is, $\dim(\I(M))=4$ and each principal orbit must be the lens space $L_{p, 1}$ for some integer $p\ge1$.
\end{corollary}

\noindent{\bf Organization.}
The paper is organized as follows. We'll start with a preliminary section and then study cohomogeneity one GRS and prove Theorem \ref{t1.1} in Section \ref{coho_one}. 
In Section \ref{S4}, we consider GKRS whose isometry group has dimension at least $n^2-1$, and establish Theorem \ref{m1.1v2}, whose proof implies Theorem \ref{m1.1}. 
The first step is to show that the assumption will lead to an action of cohomogeneity one. Then, applying Theorem \ref{t1.1}, orbits of the action coincide with level sets of the potential function $f$. The global K\"{a}hler structure indeed induces a natural, almost contact metric structure on each regular level set. The high degree of symmetry on such level sets leads to precise calculation of the shape operator and, consequently, each regular level set is a deformation of a fixed homogeneous Sasakian model, giving the desired conclusion. Section \ref{S5} is devoted to the proof of Theorem \ref{t3.1} and Corollary \ref{c1.8}. From the proof of Theorem \ref{m1.1v2}, we know that the almost contact metric structure on each regular level set is a deformation of a Sasakian model. The argument to prove Theorem \ref{t3.1} is essentially based on the relationship between Killing vector fields 
 on this Sasakian model, $P$, and their projection to the quotient leaf space, $N$.  Since the obstruction to lifting the symmetry from $N$ to $P$ is given by the first cohomology group of $N$, we may combine this information with the assumption on the dimension of $\I(M)$ to prove the result. 
\section{Preliminaries}
In this section, we gather definitions and basic results from  transformation groups, almost contact 
metric structures and gradient K\"ahler Ricci solitons, and the lifting and projection of symmetry in such spaces.
\subsection{Transformation Groups}

\quad\quad We first recall some facts about transformation groups. For more details, we refer the reader to Alexandrino and Bettiol \cite{BA}, Bredon \cite{Bre}, and Kobayashi \cite{Ko3}. Let $G$ be a group acting on a manifold $M$. For a point $p\in M$, the subgroup
$G_p:=\{\phi \in G \mid \phi \cdot p=p\} \subset G$
is called isotropy group or stabilizer of $p \in M$, and
$G \cdot p:=\{\phi \cdot p \mid \phi \in G\} \subset M$ 
is called the orbit of $p \in M$.  Where convenient, we write the orbit as $G/G_p$.
We say that an orbit is {\em principal} if its isotropy subgroup is minimal among all isotropy subgroups and, in this case, we call it the {\em principal isotropy subgroup}. We say that an orbit is   {\em singular} if its dimension is strictly less than that of a principal orbit and {\em exceptional} if the orbit is non-principal and has the same dimension as a principal orbit.
Let $M^{\reg}$ denote the regular part of $M$, corresponding to the union of all principal $G$-orbits.
$M^{\reg}$ forms an open and dense subset of $M$.

On a Riemannian manifold, an action is called isometric if it preserves the metric. We first recall the following well-known result; see Theorem 1.6.10 in Boyer and Galicki \cite{CK} and Proposition 3.62 in \cite{BA}.
\begin{lemma}\label{p1}
	Let $(M, g)$ be a connected Riemannian manifold. Then the group of isometries, $\I(M)$, is a finite-dimensional Lie group. The action of $\I(M)$ or any of its closed subgroup is proper.
\end{lemma}
 
The next lemma is a combination of Corollary 21.8 in Lee \cite{Lee} and  Proposition 3.41 and Theorem 3.49 in \cite{BA}.

\begin{lemma}\label{c2.6}
	Let $G$ be a  Lie group acting properly and isometrically on a manifold $M$. Then for all $p \in M$, the orbit $G \cdot p$ is closed in $M$, and the isotropy group $G_p$ is compact. Furthermore, the tangent space $T_p(G \cdot p)$ is $G_p$-invariant.
\end{lemma}

\noindent In particular, it is well-known that  $G_p$ acts {\em linearly} on $T_p M$ via the differential of any group element, $\phi\in G_p$, $(d \phi)_p: T_p M \rightarrow T_p M$, and is called the {\em linear isotropy representation of $\phi$}. \\

Let $(M, g, J, f, \lambda)$ be a GKRS and let $G$ denote the identity component of $\I (M)$. It is immediate that $G$ is the largest connected group of $\I(M)$. The following lemma will be crucial.
\begin{lemma}\label{c3.1}
	Suppose that $(M, g, J, f, \lambda)$ be a non-trivial GKRS. Then $G$ preserves the complex structure $J$.
\end{lemma}

\begin{proof}	By going to the universal if necessary, one can assume that the structure is irreducible. By Lemma \ref{p1}, $G$ acts properly and isometrically on $(M, g)$. For an irreducible K\"{a}hler manifold, it follows from the work of Lichnerowicz \cite{Li}  that $G$ preserves the complex structure $J$ if $n$ is odd, or if $n$ is even and the Ricci tensor Rc is non-vanishing. As the soliton considered here is non-trivial, we have that $\mathrm{Rc} \not \equiv 0$ and the conclusion follows. \end{proof}

A Lie group acting on a manifold $M$ induces a Lie algebra homomorphism from the Lie algebra of the group into the Lie algebra of smooth vector fields on the manifold. If the group preserves 
a given tensorial structure $T$ on $M$ then, $\mathcal{L}_X T=0$ where $X$ is in the image of the homomorphism. In case of a Riemannian metric, the vector field is called {\em Killing}. For an almost Hermitian structure $(M, g, J)$, the vector field is called an {\em infinitesimal automorphism}. We adopt the same terminology for an almost contact metric structure as described below.  
	
\subsection{Almost Contact Metric Structure and GKRS}
\label{acms}
\quad\quad In this subsection, we give a brief introduction to almost contact geometry and see how it arises in our setup. The main reference we use here is \cite{CK}. 
\begin{definition} 
	An odd-dimensional manifold $P$ is called {\em almost contact} if there exists a triple $(\zeta, \eta, \Phi)$, where $\zeta$ is a vector field,  $\eta$ is a $1$-form, and $\Phi$ is a tensor field of type $(1, 1)$  such that, everywhere on $P$ we have  
	\[\eta(\zeta)=1 \textrm{  and } \Phi^2=-\text{Identity}+\zeta\otimes \eta.\]
	If, in addition, there is a Riemannian metric $g$ compatible with $\Phi$, that is, 
	$$g(\Phi(X), \Phi(Y))=g(X, Y)-\eta(X)\eta(Y),$$
	then $(P, g, \zeta, \eta, \Phi)$ is called an {\em almost contact metric structure}.
\end{definition}
\begin{definition}
	A smooth vector field $X$ on $P$ is an infinitesimal automorphism of the almost contact metric structure if its Lie derivative preserves each component of the structure: 
	\begin{align*}
		\mathcal{L}_X g &= 0,~~ \mathcal{L}_X \zeta = 0, ~~{\textrm {and}} ~~	\mathcal{L}_X \Phi = 0.
	\end{align*}
	Integrating these vector fields forms the automorphism group of the structure.  
\end{definition}
\noindent It is immediate from the definition that $\zeta$, also called the {\em Reeb vector field}, is nowhere vanishing and, by the Frobenius theorem,  
generates a $1$-dimensional foliation $\mathcal{F}$. 
If $\zeta$ is a Killing vector field, then the foliation is Riemannian and the transverse geometry is significantly more tractable. Indeed, the subspace of the tangent space perpendicular to $\zeta$ is called the transverse subspace. The holomorphic sectional curvature of an almost contact metric structure refers to the sectional curvature of a $\Phi$-invariant tangent plane in the transverse subspace. It is also natural to consider the transverse metric $g^\perp$, which is the restriction of $g$ to the transverse subspace, that is, 
\[g=g^\perp +\eta\otimes \eta. \]

\begin{definition}
	An almost contact metric structure $(P, \zeta, \eta, \Phi, g)$ is called {\em contact} if $$g(X, \Phi(Y))=d\eta(X, Y).$$
	Additionally, a contact metric manifold $(P,\zeta, \eta, \Phi, g)$ is called {\em K-contact} if $\zeta$ is Killing, that is,  $\mathcal{L}_{\zeta}g=0.$
\end{definition}

\begin{definition}
	A {\em Sasakian structure} is a contact metric structure such that the cone over $P$, $(P\times \mathbb{R}^+, r^2 g+ dr^2)$ is K\"{a}hler. Moreover, the K\"{a}hler form is given by $d(r^2\eta)$. 
\end{definition}
\noindent Indeed, a Sasakian structure is $K$-contact and, furthermore, satisfies an integrability condition.\\

For a K\"{a}hler manifold, a regular level set of a non-trivial real function on $M$ admits %with 
an almost contact metric structure in the following way. Let $(M, g, J)$ be a K\"{a}hler manifold with a smooth non-trivial function $f: M\mapsto \mathbb{R}$. 
For a regular value $c\in f(M)$, denote by $\Sigma_c$ its corresponding level set and let $$V:=\frac{\nabla f}{|\nabla f|},\,\, g_c:=g_{\mid \Sigma_c},\,\,  \zeta:=-J(V),\,\,  \eta(\cdot):=g\left(\cdot, \zeta\right),\,\, \textrm{and}\,\, \Phi(\cdot):=-\eta(\cdot) V+J(\cdot).$$
Then one can verify that $\left(\Sigma_c, g_c, \zeta, \eta, \Phi\right)$ defines an almost contact metric structure on $\Sigma_c$. Furthermore, the automorphism group preserving $f$ on $(M, g, J, f)$ descends to one on the level set with the induced structure tensors by Theorem 4.6 in \cite{Hung1}. Thus, in the study of GKRS, the geometry of level sets of the potential function $f$ is intrinsically intertwined with that of the global manifold. That principle is developed in \cite{Hung0, Hung1} and the following results are important in our proof.  \\

First, we recall a transformation introduced in 
\cite{Ta}.

\begin{definition}
\label{Sdeform}
For $H, F\in \mathbb{R}^+$, a $(\pm, H, F)$-deformation is given by 
\[ \zeta^\ast =\frac{\zeta}{H}, ~~~ \eta^\ast = H\eta, ~~~ \Phi^\ast=\pm\Phi,~~~ g^\ast=  F^2 g +(H^2-F^2)\eta\otimes \eta. \]
A {\em deformed contact/Sasakian metric structure} is obtained via an $(\pm, H, F)$-deformation of a contact/Sasakian metric structure.  
\end{definition}
\begin{remark}
	It follows that the $(\pm, H, F)$-deformation preserves the automorphism group of a complete almost contact metric structure. That is, a vector field $X$ is an infinitesimal automorphism for a deformation if and only if it is an infinitesimal automorphism of the original structure.  
\end{remark}

Next, let $S$ denote the shape operator of $\Sigma_c\subset M,$
\[SX:= \nabla_X V.\] We recall Theorem F in \cite{Hung0}.
\begin{theorem}
	\label{hyperdeformedSasa}{\rm\cite{Hung0}}
	Let $(P, \eta, \zeta, \Phi, g)$ be an almost contact manifold. Then  $$S=\alpha \text{Id}+\beta \zeta \otimes \eta,$$ for constants $\alpha$ and $\beta$ if and only if $(P, \eta, \zeta, \Phi, g)$ is
	\begin{enumerate}
		\item[\rm 1.]  A deformed-Sasakian structure when $\alpha\neq 0$; and
		\item[\rm 2.]  Locally a Riemannian product when $\alpha=0$,. 
	\end{enumerate}
	
\end{theorem}
\noindent In addition, \cite{Hung0} observes that if each level set of the potential function of a GKRS is deformed contact, then the soliton structure is foliated by deformations of a corresponding Sasakian model.

\begin{theorem}{\rm\cite[Theorem D]{Hung0}}
	\label{HDcontact}
	Let $(M, g, J, f, \lambda)$ be a complete connected K\"{a}hler GRS. For each regular value $c$ of $f$, suppose that $(M_c, \zeta_{c}, \eta_{c}, \Phi_{c}, g_{c})$ is a deformed contact structure. Then the soliton is totally determined by a connected Sasakian model $(P, \eta, \zeta, \Phi, g_P)$ which is a Riemannian submersion, with circle fibers, over a K\"{a}hler-Einstein manifold $(N, g_N, J_N)$ with $\Rc_N=k g_N.$ That is, there is a submersion map $\pi: P\mapsto N$ such that
	\[g_P= \eta\otimes \eta +\pi^\ast g_N ~~\textrm{and}~~ d\eta=\pi ^\ast \omega_N.\]
	There is an interval $I$ with coordinate $s$ such that there is a diffeomorphism $$\phi: I\times P \mapsto M_o,$$ a dense subset of $M$,  $f\circ \phi= Bs+C$, and 
		\begin{equation} 
		\label{calabi}
		\phi^\ast g= \frac{ds^2}{\alpha(s)}+ \alpha(s) \eta\otimes\eta + (2s+A) \pi ^\ast g_N.\end{equation}
	Here $A, B, C$ are constant and $\alpha$ solves a first order equation, for $n=\dim_{\mathbb{C}}N$, namely  
	\[\lambda(2s+A)=k-\frac{d\alpha}{ds}-\frac{2n\alpha}{2s+A}+B\alpha. \]
There is a boundary point of $I$ such that $\alpha\rightarrow 0$. Furthermore, if $(2s+A)\rightarrow 0$ towards that end point, then $(P, \eta, \zeta, \Phi, g_P)$ is the standard Sasakian sphere and $(N, g_N, J_N)$ is, up to homothety, isomorphic to a standard complex projective space. 

\end{theorem}

\subsection{Projection and Lifting of Symmetry}
\label{liftofsymmetry}

\quad\quad The manifold $N$ 
appearing in Theorem \ref{HDcontact} actually corresponds to the space of trajectories generated by the Reeb vector field of the Sasakian model. In general, each such trajectory is called a {\em leaf}, and the leaf space provides a reduction scheme regarding the symmetry of the original structure. Precisely, let $(P, \eta, \zeta, \Phi, g)$ be a $K$-contact structure and $\mathcal{F}_\zeta$ denotes the Riemannian foliation generated by $\zeta$. We further assume that each leaf is compact. 

The foliation is said to be regular if around each point $x$, there is a foliated coordinate chart $(U, x)$ such that each leaf of $\mathcal{F}_\zeta$ passes through $U$ at most once. An implication of regularity in this context is that the action generated by $\zeta$ is free. The space of leaves is denoted by $P/{\mathcal{F}_\zeta}$ and there is a canonical projection $\pi: P\mapsto P/\mathcal{F}$, mapping each point to its leaf. It is known
that (see, for example,  Theorem 7.1.3 in \cite{CK}) if $(P, \eta, \zeta, \Phi, g)$  generates a regular foliation, then $P/\mathcal{F}_{\zeta}$ is an almost K\"{a}hler manifold. The induced structure tensors $G$ and $J$ satisfy the following relations
\begin{equation*}
	g =\pi^\ast G+ \eta\otimes \eta \,\, \textrm{and}\,\,
	\overline {JX} = \Phi \overline{X}, 
\end{equation*}
where $\overline{X}$ is the horizontal lift with respect to $\eta$. That is, the horizontal lift $\overline{X}$ is uniquely determined by the equations $\eta(\overline{X})=0$ and $\pi_\ast (\overline{X})=X.$ Furthermore, the canonical projection $\pi: (P, g) \mapsto (P/\mathcal{F}, G)$ is a Riemannian submersion and, letting $W$ denote the K\"{a}hler form on the latter, we have  
\[ d\eta = \pi^\ast W.\]

\begin{remark} 
	More generally, one can replace regularity by quasi-regularity and allow the leaf space to have orbifold singularities.  
\end{remark}

The canonical projection induces a map from the Lie algebra of infinitesimal automorphisms on $P$ to one on $N$ by Lemma 8.17 in \cite{CK}. Furthermore, since $\pi$ is a Riemannian submersion, the kernel of this map is precisely the span of $\zeta$. Thus, the automorphism group of $(P, g, \zeta, \eta, \Phi)$ is of dimension at most that of $(P/\mathcal{F}_\zeta, G, J)$ plus one as observed in \cite{Ta}. However, this map might not be onto. Indeed, the lifting of an infinitesimal automorphism from $P/\mathcal{F}_\zeta$ to $P$ 
has a topological obstruction as described below.  

Let $X$ be an infinitesimal automorphism of $(P/\mathcal{F}_\zeta, G, J)$. A lift of $X$ is a vector field on $P$ of the form $\overline{X}+ \overline{a} \zeta$ where $\overline{a}$ is a smooth function which is constant on each leaf; that is, $\overline{a}$ is a lift of a function $a: P/\mathcal{F}_\zeta \mapsto \mathbb{R}$. By direct 
calculation, one sees that $\overline{X}+ \overline{a} \zeta$ is an infinitesimal automorphism of $(P, g, \Phi, \eta, \zeta)$ if and only if $\mathcal{L}_{\overline{X}+ a \zeta}\eta =0$ or, equivalently, if there is a smooth function $a: P/\mathcal{F}_\zeta \mapsto \mathbb{R}$ such that $da (\cdot)= W(X, \cdot)$ by Corollary 8.1.9 in \cite{CK}. 
In other words, the cohomology class of the $1$-form $W(X, \cdot)$ is vanishing. We summarize this discussion in the following result.
\begin{lemma}\label{liftsym}
	Let $(P, g, \Phi, \eta, \zeta)$ be a regular $K$-contact structure with compact leaves and $(P/{\mathcal{F}_\zeta}, G, J)$ be the leaf space with induced structure tensors. The canonical projection $\pi: P\mapsto P/{\mathcal{F}_\zeta}$ induces a linear map between the corresponding Lie algebras of infinitesimal automorphisms. The kernel of this map is the span of $\zeta$. An infinitesimal automorphism $X$ of $(P/{\mathcal{F}_\zeta}, G, J)$ is in the image of this map if and only if the cohomology class of the $1$-form $W(X, \cdot)$ vanishes.     
\end{lemma}

\section{Cohomogeneity One Preserves the Potential Function}	\label{coho_one}
	
\quad\quad In this section, we show that, for a non-trivial GRS with $\I(M)$ acting by cohomogeneity one, the potential function $f$ is preserved by the action. When $\I(M)$ is compact, this is immediate by an averaging argument, see p. 263 of  Dancer and Wang \cite{DaW}. In the case $\lambda \neq 0$, the result follows from Corollary 2.2 of \cite{Pete2}, which we restate here for the sake of completeness. 

\begin{proposition}{\rm \cite{Pete2}}\label{p2.1}
	If $X$ is a Killing field on a GRS $\left(M, g, f, \lambda\right)$, with 
	$\lambda \neq 0$ then either $D_X f=0$ or the universal cover of $M$ splits isometrically as $\widetilde{M}=N \times \mathbb{R}$, where $N$ is a GRS with the same radial curvatures as $M$.
\end{proposition}

Therefore, it suffices to prove Theorem \ref{t1.1} in the case $\lambda=0$, which reduces to a proof of the following result.
\begin{theorem}\label{t3.3}
	
	Let $\left(M, g, f\right)$ be a non-trivial steady GRS of real dimension $n \geq 3$. Assume that the action of $\I(M)$ is by cohomogeneity one. Then the potential function $f$ is invariant under the action.
\end{theorem}

Before we prove Theorem \ref{t3.3}, we establish a number of important lemmas.
We begin with a basic observation relating a Killing vector field and the potential function $f$, whose proof follows directly from Proposition \ref{p2.1}  and Lemma 2.6 in \cite{DT}.

\begin{lemma}\label{l3.1}
	Let $\left(M, g, f\right)$ be a non-trivial GRS of dimension $n \geq 3$. If $X$ is a Killing vector field then $g(X,\nabla f)$ is constant and $[X, \nabla f]=0$. 
\end{lemma}

If $\left(M, g, f\right)$ is, additionally, of cohomogeneity one and not homogeneous then, a neighborhood around a principal orbit is diffeomorphic to $I\times P$, for an interval $I$ with coordinate $t$, and a differentiable manifold $P$ such that the the metric can be written as  $g=dt^2+ g_t$ by Proposition 1.4 in Podesta \cite{Po}. Here $g_t$ is a family of $\I(M)$-invariant metrics on $P$.
 
\begin{lemma}\label{basis2}
	Let $\left(M, g, f\right)$ be a non-trivial GRS, and let $\I (M)$ act on $M$ by cohomogeneity one. 
	Assume that $\dim (\mathfrak{I}(M))=k$ for $1\leq k<\infty,$ where $\mathfrak{I}(M)$ is the Lie algebra of $\I (M).$ Suppose there exists a basis $\left\{Z_1, Z_2, \ldots, Z_k\right\}$ for $\mathfrak{I}(M)$ such that 
	\begin{align*}
		g(\nabla f, Z_1)=a\neq 0 ~~\text{and   }	g\left(Z_i, \nabla f\right)=0 \,\,\text{for all}\,\, 2\leq i\leq k.
	\end{align*} 
	Then, on a neighborhood around a principal orbit, 
		$$Z_1	\perp \text{span}\{Z_2, \ldots, Z_k\} \,\,\,\text{and}\,\,\, \nabla f=u \partial_t+u_1Z_1$$
for some smooth functions $u, u_1$.
\end{lemma}
\begin{proof}%[Proof of Lemma \ref{basis2}]
On $M^{\reg}$, the tangent bundle splits orthogonally as
$$T M=\operatorname{span}\left\{\partial_t\right\} \oplus T P.$$
Since the Killing vector fields $Z_i$ are tangent to the orbits, we have
$
T P=\operatorname{span}\left\{Z_1, \ldots, Z_k\right\}.
$
The assumption $g\left(\nabla f, Z_i\right)=0$ for all $i \geq 2$ implies that
$$
\nabla f \perp \operatorname{span}\left\{Z_2, \ldots, Z_k\right\} .
$$
Since $\partial_t$ is orthogonal to $T P$, the tangential component of $\nabla f$ must lie in the direction of $Z_1$. Hence there exist functions $u, u_1$ such that
$$
\nabla f=u \partial_t+u_1 Z_1.
$$
Taking the inner product with $Z_1$, we obtain
$$
a=g\left(\nabla f, Z_1\right)=u_1\left|Z_1\right|^2,
$$
and therefore $u_1=\frac{a}{\left|Z_1\right|^2} \neq 0$. Since $g\left(\partial_t, Z_j\right)=0$ for any $j \geq 2$, we get 
$$
0=g\left(\nabla f, Z_j\right)=g\left(u \partial_t+u_1 Z_1, Z_j\right)=u_1 g\left(Z_1, Z_j\right).
$$
Because $u_1 \neq 0$, it follows that $g\left(Z_1, Z_j\right)=0$ for all $2 \leq i \leq k$.
\end{proof}
\begin{lemma}\label{basis3}  Under the same assumptions as in Lemma \ref{basis2}, if $\lambda=0,$ then
	$\left.u\right|_{\left(P, g_t\right)},$ $\left.u_1\right|_{\left(P, g_t\right)}$ and $\left.|Z_1\right|_{\left(P, g_t\right)}$ are constants for each $t.$
\end{lemma}
\begin{proof}
	We compute, for a tangential vector $X$, 
	\begin{align*}
		g(\left[Z_i, \partial_t\right], X) &= g(\nabla_{Z_i}\partial_t, X)-g(\nabla_{\partial_t}Z_i, X)\\
		&= g(\nabla_{Z_i}\partial_t, X)+g(\nabla_X Z_i, \partial_t)\\
		&= g(\nabla_{Z_i}\partial_t, X)-g(\nabla_X \partial_t, Z_i)=0.
	\end{align*}
	Here the second equality follows because $Z_i$ is a Killing vector field; the third equality follows because $Z_i$  is tangential and the last equality is due to the fact that $\partial_t$ is a normal vector to the principal orbit, which is a hypersurface. 
	
	Similarly, $g(\left[Z_i, \partial_t\right], \partial_t)=0$ and we conclude that $\left[Z_i, \partial_t\right]=0$ for all  $1\leq i\leq k$. By Lemma \ref{l3.1}, 
		\begin{align}\label{3.9}
		0=\left[Z_i, \nabla f\right]=\left[Z_i, u \partial_t+u_1 Z_1\right] =\left(\nabla_{Z_i} u\right) \partial_t+\left(\nabla_{Z_i} u_1\right) Z_1+ u_1[Z_i, Z_1] .
	\end{align}
	Taking the inner product with $\partial_t$ yields  
	$$0=\left(\nabla_{Z_i} u\right) g\left(\partial_t, \partial_t\right)+u_1 g\left(\left[Z_i, Z_1\right], \partial_t\right)+\left(\nabla_{Z_i} u_1\right) g\left(Z_1, \partial_t\right)=\nabla_{Z_i} u,$$
	for all $1\leq i\leq k$.	Since $TP$ is spanned by $\left\{ Z_1, Z_2, \ldots, Z_k\right\},$ it follows that $\left.u\right|_{\left(P, g_t\right)}$ is constant for each value of $t$.

	By the soliton equation, $\scal-|\nabla f|^2-2\lambda f$ is constant on the manifold, where $\scal$ denotes the scalar curvature. Since the scalar curvature is invariant under an isometric action, if $\lambda=0$,  then $|\nabla f|$ is also invariant under the isometric action. Therefore, 
	\[
	|\nabla f|^2= u^2+u_1^2\left|Z_1\right|^2
	\] is constant on each $t$. Since $\left.u\right|_{\left(P, g_t\right)}$ is constant for each $t$, so is $\left.\left(u_1^2\left|Z_1\right|^2\right)\right|_{\left(P, g_t\right)}$. Moreover, 
	$$
	a=g\left(\nabla f, Z_1\right)=g\left(u \partial_t+u_1 Z_1, Z_1\right)=u_1\left|Z_1\right|^2.
	$$
	Therefore, each $\left.u_1\right|_{\left(P, g_t\right)}$ and $\left.|Z_1\right|_{\left(P, g_t\right)}$ are constant for each value of $t$.   
\end{proof}

Now we are ready to prove Theorem \ref{t3.3}.

\begin{proof}
	Suppose that $\dim \mathfrak{I}(M)=k$ for $1\leq k<\infty,$ where $\mathfrak{I}(M)$ is the Lie algebra of $\I (M).$ Notice that if $g(Z,\nabla f)=0$ for all $Z\in \mathfrak{I}(M),$ then $f$ is invariant under the action of $\I(M)$ and we are done. Therefore, we may assume that there exists a Killing field $\widetilde{Z}_1$ such that
	$g(\widetilde{Z}_1,\nabla f)\neq 0$.  Then there exists a basis $\left\{\widetilde{Z}_1, \widetilde{Z}_2, \ldots, \widetilde{Z}_k\right\}$ for $\mathfrak{I}(M)$ and we may set \begin{center}
		$a=g\left(\widetilde{Z}_1, \nabla f\right), b_i=g\left(\widetilde{Z}_i, \nabla f\right)$ for all $i\geq 2.$
	\end{center} By Lemma \ref{l3.1}, $a$ and each $b_i$ are constant. Moreover, we have $a\neq 0.$ Now consider the following set $\left\{{Z}_1, {Z}_2, \ldots, {Z}_k\right\}$ where $Z_1=\widetilde{Z}_1$ and 
	$
	Z_i=\widetilde{Z}_i-\frac{b_i}{a} Z_1,\,\forall i \geq 2.
	$
	Then, the set $\left\{{Z}_1, {Z}_2, \ldots, {Z}_k\right\}$ is also a basis for $\mathfrak{I}(M)$ and
	\begin{align}
		g\left(Z_i, \nabla f\right)&=g\left(\widetilde{Z}_i-\frac{b_i}{a} Z_1, \nabla f\right)
		=b_i-\frac{b_i}{a}\cdot a=0,\nonumber
	\end{align}
for all $i \geq 2$.
Using Lemma \ref{basis2}, in a neighborhood around a principal orbit, we have 
$$Z_1	\perp \text{span}\{Z_2, \ldots, Z_k\} \,\,\,\text{and}\,\,\, \nabla f=u \partial_t+u_1Z_1$$
for some smooth functions $u, u_1.$ From Lemma \ref{basis3}, we get that $\left.u\right|_{\left(P, g_t\right)},$ $\left.u_1\right|_{\left(P, g_t\right)}$ and $\left.|Z_1\right|_{\left(P, g_t\right)}$ are constants for each $t.$
	The Codazzi equation and the cohomogeneity one condition imply that, for any tangential vector $X$, $2\Rc(\partial_t, X)=0$. The soliton structure (\ref{grs}) then reads
	\begin{align*}	
		 0 &= (\mathcal{L}_{\nabla f}g)(\partial_t, X)\\
		&= (\mathcal{L}_{u \partial_t} g)(\partial_t, X)+u_1(\mathcal{L}_{Z_1}g)(\partial_t, X)+(\nabla_{\partial_t} u_1)g(Z_1, X)+(\nabla_X u_1)g(Z_1, \partial_t).
	\end{align*}
	Since $Z_1$ is Killing and $u, u_1$ only depending on $t$, we obtain
	$0=(\nabla_{\partial_t} u_1)g(Z_1, X).$
	Since $X$ is an arbitrary tangential vector field and $Z_1$ is nowhere vanishing, we conclude that $\nabla_{\partial_t} u_1=0$. This, in addition to the fact that principal orbits make up a dense subset of $M$, shows that $u_1$ is constant. 
	Define a function $\widetilde{f}: I\times P \to \mathbb{R}$  as 
	$\widetilde{f}=\int u \,dt$. Then, we have \begin{center}
		$\nabla \widetilde{f}=u \partial_t$\,\, and\,\, $\nabla f=\nabla \widetilde{f}+u_1 Z_1.$
	\end{center} Since $u_1 Z_1$
	is a Killing vector field, we get 
	$$
	\Hess f=\Hess \widetilde{f}+\frac{1}{2} \mathcal{L}_{u_1 Z_1} g=\Hess \widetilde{f}.
	$$
	 Let $h= f-\widetilde{f}$, then $\Hess h = 0$, so $\nabla h$ is a parallel vector field. Since $M$ is locally irreducible, $\nabla h \equiv 0\equiv u_1 Z_1$ on the neighborhood $I\times P$.  
As $u_1$ is constant, the identity $u_1 Z_1 \equiv 0$ implies that $u_1=0$. This shows that
$$
a=g\left(Z_1, \nabla f\right)=u_1\left|Z_1\right|^2
=0,$$ which is a contradiction. Hence $f$ is invariant under the action of $\I(M)$, as desired.
\end{proof}

	\section{Almost Maximal Isometry Group}\label{S4}

\quad\quad	
Our objective in this section is to establish the following theorem whose proof implies Theorem \ref{m1.1}.

\begin{theorem}\label{m1.1v2}
		Let $(M, g, J, f, \lambda)$ be a non-trivial GKRS of complex dimension $n\geq 2.$ Suppose that the isometry group of $M$ is of dimension at least $n^2-1$.  Then, the action is of cohomogeneity one and preserves the potential function, $f$. 
		
		Further, the soliton  $M$ is  the smooth compactification at one or both of the ends of $P\times I$, where $I$ is an open interval and $
		\left(P, \eta, \zeta, \Phi, g_P\right)$ is a connected, homogeneous Sasakian manifold 
		that is a
		Riemannian submersion, with circle fibers, over a homogeneous K\"{a}hler-Einstein manifold $\left(N, g_N, J_N\right)$, satisfying the following properties.
		\begin{enumerate}
			\item[\rm 1.] There is a submersion map $\pi: P \mapsto N$ such that
			$$
			g_P=\eta \otimes \eta+\pi^* g_N\quad \textrm{and}\quad d\eta=\pi^* \omega_N.
			$$
			\item[\rm 2.] The open interval $I \subset \mathbb{R}$ is parametrized by $s$ and there exists a diffeomorphism
			$
			\phi: P \times I \longrightarrow M^{\mathrm{reg}},
			$ where $M^{\mathrm{reg}}$ is the regular part of $M$, such that, for constants $A, B, C$,
			$$
			f \circ \phi=B s+C
			$$
			and 
			$$
			\phi^* g=\frac{d s^2}{\alpha(s)}+\alpha(s) \eta \otimes \eta+(2 s+A) \pi^* g_N,$$
			 where 
$\alpha: I \rightarrow(0, \infty)$. 
		
			\item[\rm 3.] Letting $k$ be the Einstein constant of $(N, g_N)$, the function $\alpha$ satisfies  the following first-order ODE:
			\begin{align}\label{o1.1}
				\lambda(2s + A)
				= k - \alpha'(s) - \frac{2n\,\alpha(s)}{2s + A} + B\,\alpha(s),
				\quad \text{for all } s \in I.
			\end{align}		
					\end{enumerate}
In particular, there exists at least one boundary point of $I$ for which $\alpha(s) \rightarrow 0$. 

	\end{theorem}

Before proving Theorem \ref{m1.1v2}, we first establish some facts. Since $g$ is a K\"{a}hler metric, we have $J^2=-{\rm Id}$  and $g(J X, J Y)=g(X, Y)$ for all vector fields $X, Y$. Thus,
\begin{align}
	\label{m3.1}
	g(\nabla f, J \nabla f) &=-g\left(J^2(\nabla f), J \nabla f\right)=-g(J \nabla f, \nabla f)=-g(\nabla f, J \nabla f)=0.
\end{align}
We consider the following vector subspaces at a point $p$,
\begin{align*}
	\mathcal{W} &: =\operatorname{span}_{\mathbb{R}}\left\{\left.(\nabla f)\right|_p,\left.(J \nabla f)\right|_p\right\},\\
	\mathcal{W}^{\perp} &:=\left\{X_p \in T_p M \mid g_p\left(X_p, w\right)=0, \forall w \in \mathcal{W}\right\}.
\end{align*}
It follows that, for $p$ in a dense subset of $M$ where $\nabla f\neq \vec{0}$, $T_p M=\mathcal{W} \oplus \mathcal{W}^{\perp}$, $\dim_{\mathbb{R}} \mathcal{W}=2$, and  $\dim_{\mathbb{R}} \mathcal{W}^{\perp}=2n-2$. Recall that by Lemma \ref{c3.1}  the identity component of $\I(M, g)$, which we denote by $G$, also preserves $J$. 
 
\begin{lemma}\label{n3} 	With the hypotheses as in Theorem \ref{m1.1v2}, at
	a point $p$ where $\left.(\nabla f)\right|_p\neq 0$, for all $\phi\in {G}_p$, we have that
	$\left.(d \phi)_p\right|_{\mathcal{W}}=\operatorname{Id}_{\mathcal{W}}$ and $\mathcal{W}^{\perp}$ is  ${G}_p$-invariant.
\end{lemma}
\begin{proof}
	Since $\phi$ is an isometry, we get
	$$
	0=\phi^*(\Rc+\Hess f)=\phi^*(\Rc)+\phi^*(\Hess f)=\Rc+\Hess\left(\phi^* f\right)=\Rc+\Hess(f \circ \phi).
	$$
	Subtracting the soliton equation for $f$ yields
	$
	\Hess(f\circ\phi - f) = 0.$
	Let $h= f\circ\phi - f$, then $\Hess h = 0$, so $\nabla h$ is a parallel vector field.  An argument analogous to that in the proof of Theorem \ref{t3.3} yields $\nabla h=0$ and, since $M$ is connected, it follows that $h$ is constant. Hence, we get
	$
	f \circ \phi=f+C(\phi),
	$
	for some constant $C(\phi) \in \mathbb{R}$ depending on $\phi$.	In particular, for $\phi\in G_p$, $\phi(p)=p$ and so $h(p)\equiv 0$ and thus, $C(\phi)=0$. Thus, $f$ is invariant by $\phi$ and, consequently,
	$$
	(d \phi)_p\left(\left.(\nabla f)\right|_p\right)=\left.(\nabla f)\right|_p.$$
	By Lemma \ref{c3.1}, $\Phi$ also preserves $J$ and we have:
	$$
	\begin{aligned}
		(d \phi)_p\left(\left.(J \nabla f)\right|_p\right)=\left((d \phi)_p \circ J_p\right)\left(\left.(\nabla f)\right|_p\right) =\left(J_p \circ(d \phi)_p\right)\left(\left.(\nabla f)\right|_p\right)=\left.(J \nabla f)\right|_p .
	\end{aligned}
	$$
	These identities yield that $(d \phi)_p$ acts as the identity on $\mathcal{W}$, that is, $\left.(d \phi)_p\right|_{\mathcal{W}}=\operatorname{Id}_\mathcal{W}.$

	Now, take any $v \in \mathcal{W}^{\perp}$. For any $w \in \mathcal{W}$, since $\phi$ is an isometry and satisfies $\left.(d \phi)_p\right|_{\mathcal{W}}=\operatorname{Id}_{\mathcal{W}}$, we have
	$$
	g_p\left((d \phi)_p(v), w\right)=g_p\left((d \phi)_p(v),(d \phi)_p(w)\right)=g_p(v, w)=0 .
	$$
	Hence $(d \phi)_p(v)$ is orthogonal to every $w \in \mathcal{W}$, which means $(d \phi)_p(v) \in \mathcal{W}^{\perp}$. As $\phi$ is arbitrary, the conclusion follows.  
\end{proof}

From Lemma \ref{c2.6}, we deduce that the restriction $\left.(d \phi)_p\right|_{\mathcal{W}^{\perp}}: \mathcal{W}^{\perp} \rightarrow \mathcal{W}^{\perp}$ is a linear automorphism. Furthermore, since $\phi$ preserves $g$ and $J$, $\left.(d \phi)_p\right|_{\mathcal{W}^{\perp}}$ lies in the unitary group of the Hermitian space $(\mathcal{W}^{\perp}, g, J)$, that is, 
$$
\left.(d \phi)_p\right|_{\mathcal{W}^{\perp}} \in U\left(\mathcal{W}^{\perp}\right):=\left\{A \in {\rm{Gl}}\left(\mathcal{W}^{\perp}\right) \mid A^* g=g, A J=J A\right\}.
$$
We next consider the following isotropy representation
\begin{equation}\label{h3.3}
	\begin{aligned}
		\vartheta_p: G_p & \rightarrow U\left(\mathcal{W}^{\perp}\right) \\
		\phi & \mapsto \vartheta_p(\phi)=\left.(d \phi)_p\right|_{\mathcal{W}^{\perp}}.
	\end{aligned}
\end{equation}
\begin{lemma}\label{n5}
	Let the hypotheses as in Theorem \ref{m1.1v2}, then the isotropy representation $\vartheta_p$ is a Lie group homomorphism and is faithful.  
	Moreover,  
	$\dim(G \cdot p) \geq 2n-2.$
\end{lemma}
\begin{proof}
	It is straightforward to verify that $\vartheta_p$ is a Lie group homomorphism. 
	Furthermore, we have 
	$
	\ker\left(\vartheta_p\right)=\left\{\phi \in G_p\left|(d \phi)_p\right|_{\mathcal{W}^{\perp}}=\operatorname{Id}_{\mathcal{W}^{\perp}}\right\}.
	$
	For $\phi \in \ker\left(\vartheta_p\right),$ we get $\left.(d \phi)_p\right|_{\mathcal{W}^{\perp}}=\operatorname{Id}_{\mathcal{W}^{\perp}}$. This combined with the fact that $\left.(d \phi)_p\right|_{\mathcal{W}}=\operatorname{Id}_{\mathcal{W}}$ implies $(d \phi)_p=\operatorname{Id}_{T_p M}.$ Consequently, 
	$\phi$ must be the identity map, giving us that  $\vartheta_p$ is injective.

	Since $\dim_{\mathbb{C}} \mathcal{W}^{\perp}=n-1$, $U(\mathcal{W}^{\perp})$ is isomorphic to the unitary group
	$U(n-1)$, and in particular,  $\dim G_p \leq(n-1)^2$. Recall that we assume that $\dim G\geq n^2-1$, so 
	we obtain
	$$
	\dim(G \cdot p)=\dim G-\dim G_p \geq n^2-1-(n-1)^2=2n-2 .
	$$
\end{proof}

\begin{proposition}\label{m2.7} With the hypotheses as in Theorem \ref{m1.1v2}, if $\dim(G)\geq n^2-1$ then $G$ acts by cohomogeneity one.  In particular, the potential function $f$ is invariant under the $G$-action.
\end{proposition}
\begin{proof} Remark that once we show that $G$ acts by cohomogeneity one, we can apply Theorem \ref{t1.1} showing that $f$ is invariant under the $G$-action. 
Let $p$ be an arbitrary point on a regular level set $\Sigma_c= \{f=c\}$.  
    	By Lemma \ref{n5}, $\dim(G \cdot p)\geq 2n-2$ and we want to prove that $\dim(G \cdot p) = 2n-1$. Theorem 1.1 in \cite{Pete2} establishes that homogeneous GRS are rigid. 
    	Since we assume the soliton is non-trivial, it follows that the $G$-action is not transitive and so $\dim(G \cdot p)<2n$. We now  assume that $\dim(G \cdot p) = 2n-2$ to derive a contradiction. In particular, we see that 
	$$
	\dim G_p=\dim G-\dim(G\cdot p)\geq (n^2-1)-(2n-2)=(n-1)^2.$$
	Recall $\vartheta_p : G_p \longrightarrow U(\mathcal W^\perp)\simeq U(n-1)$ is injective by Lemma \ref{n5}. Since $U(n-1)$ is connected and $\dim G_p\geq \dim U(n-1)$, one concludes that $
	\vartheta_p\left(G_p\right)\cong U(n-1) .
	$

Now, it is well-known that 
	when $(M, g, J)$ is K\"ahler and $\Hess f$ is $J$-invariant, then $J\nabla f$  is a Killing vector field, see, for example, Lemma 2.2 of \cite{Hung0}.  Thus, $J\nabla f\in T_p(G\cdot p)$. However, $G_p$ fixes $\mathcal{W}$ pointwise, and since our assumption is that $\dim(T_p(G\cdot p)=2n-2$, this means that at least one dimension of this tangent space is fixed by $G_p$. Since  the lowest-dimensional irreducible representation of $G_p\cong U(n-1)$ is $(2n-2)$-dimensional, this gives us a contradiction and hence $\dim((T_p(G\cdot p)=2n-1$, as desired.
	\end{proof}

By Proposition \ref{m2.7}, the potential function $f$ is invariant under the action of $G$. As described in Subsection \ref{acms}, a regular set of $f$ can be considered as an almost contact metric structure with natural induced tensors $\left(\Sigma_c, g_c, \zeta, \eta, \Phi\right)$. Furthermore, the group of automorphisms on $(M, g, J, f)$ descends to one on $\left(\Sigma_c, g_c, \zeta, \eta, \Phi\right)$ by Theorem 4.6 in \cite{Hung1}.  In particular, it preserves the tensors $\zeta, \eta$, and $\Phi$ on $\Sigma_c$.

For all vector fields $X, Y\in T\Sigma_c,$ we have, for $|\nabla f|$ constant along $\Sigma_c$, 
\begin{align}\label{q4.13}
	(\Hess f)(X, Y) & =g\left(\nabla_X \nabla f, Y\right) =g\left(\nabla_X(|\nabla f| V), Y\right)\nonumber\\
	&=X(|\nabla f|) g(V, Y)+|\nabla f| g\left(\nabla_X V, Y\right) \nonumber\\
	& =|\nabla f| g\left(\nabla_X V, Y\right)=|\nabla f| g\left(SX, Y\right).
\end{align}
Consequently, since $\Hess f$ is $J$-invariant, 
\begin{equation}\label{q4.16} S\circ J=J\circ S \,\, \text{on }\mathcal D, \end{equation}
where $\mathcal{D}:=\ker \eta=\{X \in T \Sigma_c \mid \eta(X)=g(\zeta, X)=0\}\subset T\Sigma_c$ is the transverse subspace. Note that 
$\Phi X=J X$, for all $X \in \mathcal{D}.$

\begin{lemma}\label{key}
	\label{shapeoperator} With the hypotheses as in Theorem \ref{m1.1v2}, locally, there exist functions $\alpha, \beta$ depending on values of $f$ such that 
	$$ S=\alpha\,\mathrm{Id}+\beta\,\zeta\otimes\eta . $$
\end{lemma}
\begin{proof}
	First, as $f$ is invariant under the cohomogeneity one action, integral curves of $\nabla f$ are reparametrized geodesics. Thus, $\nabla f$ is an eigenvector of $\Hess f$. Using the fact that  $\Hess f$ is $J$-invariant, for $X \in \mathcal{D}$ we have
	$$
	(\Hess f)(X, \zeta)=(\Hess f)(J X, J \zeta)=(\Hess f)\left(JX, \frac{\nabla f}{|\nabla f|}\right)=0.
	$$
	Equation \eqref{q4.13} then implies that 
	$g\left(SX, \zeta\right)=0$, for all $X \in \mathcal{D}.$
	It follows that $S(\mathcal D)\subset\mathcal D$ and $S\zeta=\mu\,\zeta$, where 
	\begin{align*}
		\mu &= g(S \zeta, \zeta)= \frac{1}{|\nabla f|}\Hess f(\zeta, \zeta)= \frac{1}{|\nabla f|^2}\nabla_{\nabla f}|\nabla f|,
	\end{align*}
	depends only on values of $f$. 
	
	We now determine $\left.S\right|_{\mathcal{D}}$ by considering the following two 
	cases: Case 1, where $n=2$, and Case 2, where $n\geq 3$.\\
	
	\noindent\textbf{Case 1}: $n=2$. Then $\dim_{\mathbb C}\mathcal D=1$.  Display \eqref{q4.16} gives us that $S|_{\mathcal D}$ is complex linear and, since $S$ is self-adjoint, $S|_{\mathcal D}=\alpha\,\mathrm{Id}_{\mathcal D}$ for 
	\begin{align*}
		\alpha &= \frac{1}{2}[\text{trace}(\Hess f)- 2\mu]=\frac{1}{2}(\Delta f- 2\mu).
	\end{align*}
	We see from the soliton equation, $\Delta f=2\lambda-\scal$, that $\Delta f$ depends only on values of $f$. Thus, so does $\alpha$.

	\noindent\textbf{Case 2}: $n\geq 3$. Since $\dim G\geq n^2-1$ and $\dim\Sigma_c=2n-1$, we have \[ \dim G_p\geq (n^2-1)-(2n-1)=(n-1)^2-1. \] Since $G$ preserves $g$, $J$, $\zeta$, $\eta$, and $\Phi$, the isotropy group $G_p$ acts on $\mathcal D_p$ by unitary transformations. Hence,  $G_p\subset U(n-1).$ Note that a connected subgroup of $U(n-1)$ of dimension at least $(n-1)^2-1$ is conjugate to $SU(n-1)$ or $U(n-1)$, either of which acts by a standard irreducible representation on ${\mathcal{D}_p}\cong \mathbb C^{n-1}$. Since $S$ is $G$-invariant, $\left.S_p\right|_{\mathcal{D}_p}$ commutes with the isotropy representation of $G_p$. Since the $G_p$-action on $\mathcal{D}_p$ is irreducible and complex, Schur's lemma implies that
	$$
	\left.S_p\right|_{\mathcal{D}_p}=\alpha(p)\operatorname{Id}_{\mathcal{D}_p},
	\quad \alpha(p)\in\mathbb{C}.
	$$
	Since $S$ is self-adjoint, the scalar $\alpha(p)$ must be real. Furthermore, as the $G$-action is transitive on $\Sigma_c$, $\alpha$ depends only on values of $f$. 

	Therefore, in both cases, there exist smooth functions $\alpha$ and $\mu$ such that $$S X=\alpha X \,\,(\forall X\in\mathcal D), \qquad S\zeta=\mu\,\zeta.$$ The result then follows for $\beta:=\mu-\alpha$. \end{proof}

We are now ready to prove the main result of this section. 

\begin{proof}[Proof of Theorem \ref{m1.1v2}]
	Proposition \ref{m2.7} gives us that the action of $G$ on $M$ is by cohomogeneity one and that the potential function $f$ is invariant under the $G$-action. Thus, the level sets of $f$, $\Sigma_c$, correspond to principal orbits of the $G$-action. 
	 Let $\Sigma_c$ be a connected component of a level set of $f$ corresponding to a regular value. By Lemma \ref{shapeoperator} and Theorem \ref{hyperdeformedSasa}, $\left(\Sigma_c, g_c, \zeta, \eta, \Phi\right)$ is either a deformed Sasakian structure or locally a Riemannian product depending on whether, for $\alpha$ as in Lemma \ref{shapeoperator}, $\alpha\neq 0$ or $\alpha=0$, respectively. By continuity, nearby regular connected components must all be of the same diffeomorphism type. In case $\alpha\neq 0$, Theorem \ref{HDcontact} then gives the desired result.

		Thus, we are left to consider the case where $\alpha= 0$. Then each $\zeta_c$ is a parallel vector field on $\Sigma_c$ with respect to the induced metric, which is locally a Riemannian product. Since $G$ acts by cohomogeneity one on $M$, the metric around a principal orbit can be written as, 		$$g =dt^2+ g_t, \nonumber$$
	for $t$ the parameter of a geodesic perpendicular to each principal orbit.
Since each $g_t$ is locally a Riemannian product with parallel vector field $\zeta=-J\frac{\nabla f}{|\nabla f|}$,
	\begin{align*}
		g_t &= H^2(t)\eta\otimes \eta+ g^\perp_t,
	\end{align*} 
	where $H(t)$ is a real-valued function depending on $t$ only. By comparing covariant and ordinary derivatives and observing that $\frac{\partial}{\partial_t}$ is a unit normal vector to the hypersurface, we have,
	\[\ddt g^\perp_t =2g\circ S .\]
	Since $\alpha=0$, $S$ restricted to the transverse subspace vanishes. Thus, 
	\[\ddt g^\perp_t =2g\circ S =0.\]
	That is, $g^\perp_t$ is independent of $t$ and the metric is a local Riemannian product, a contradiction to being locally irreducible. 
	\end{proof}

	\section{Dimension Four}\label{S5} 
	
	\quad	  
Let $(M, g, J, f, \lambda)$ be a non-trivial GKRS of complex dimension two. Suppose that $\dim \I(M,g)\geq 3$. In this section, we show that $(M, g, J, f, \lambda)$ is constructed from a connected homogeneous Sasakian manifold of constant holomorphic sectional curvature and it is of maximal symmetry; that is, the isometry group is of dimension four.

As usual, let $G$ denote the identity component of the isometry group and it follows $\dim G \geq 3$. For a regular value $c\in f(M)$, following the construction in Subsection \ref{acms}, $\left(\Sigma_c, g_c, \zeta, \eta, \Phi\right)$ defines an almost contact metric structure on $\Sigma_c$. From Theorem \ref{m1.1v2}, we know that $f$ is invariant by the action of $G$. Furthermore, each level set of $f$ is connected and each regular one with the induced almost contact metric structure is a deformation, in the sense of Definition \ref{Sdeform}, of a Sasakian manifold. Furthermore, this Sasakian structure is a Riemannian submersion with circle fibers and the base is a K\"{a}hler-Einstein manifold with constant Ricci curvature, $(N, g_N)$. That is, there is a submersion map $\pi: P \mapsto N$ such that
	$$
	g_P=\eta \otimes \eta+\pi^* g_N\,\,\,\textrm{and}\,\,\, d \eta=\pi^* \omega_N.
	$$ Consequently, the  metric on the manifold can be written as 	in Theorem \ref{m1.1v2}
		$$
	g_s=\alpha(s) \eta\otimes \eta+\beta(s) \pi^* g_N,\,\,\, \textrm{with} \,\,\,\beta(s)=2 s+A>0,$$
	where  $s$ corresponds to the parameter of the orbit space. 
	
	\begin{lemma}\label{l1} The Sasakian model described above has constant holomorphic sectional curvature. 
	\end{lemma}
	\begin{proof} 		As we see above, the transverse metric on each regular level set is a multiple of the K\"{a}hler-Einstein metric. Since $M$ is of real dimension four, the K\"{a}hler-Einstein manifold, $N$,  is of real dimension two and constant Ricci curvature on $N$ implies  the sectional curvature of $N$ is constant. Furthermore, on any Sasakian manifold, the curvature tensor satisfies an identity relating the holomorphic sectional curvature to the transverse curvature by Lemma 7.3.13 in \cite{CK}. Thus, the Sasakian structure from Theorem \ref{m1.1v2} has constant holomorphic sectional curvature.\\
	\end{proof}
	
	Proposition 4.1 of Tanno \cite{Ta2}, classifies complete, simply connected Sasakian manifolds of constant holomorphic curvature. 
	In particular, we see that the universal cover of $P$ must be precisely one of the following Sasakian space forms: 
	$$ \mathrm{SU}(2)\cong S^3 \,\,(k>0), \quad \mathrm{Nil}^3 \,\, (k=0), \,\,\, \textrm{and} \,\,\, \mathrm{\widetilde{SL}}(2,\mathbb{R}) \,\,\, (k<0),$$
	where $\kappa$ is the sectional curvature of the transverse metric. Topologically, the group $\mathrm{SU}(2)$ is diffeomorphic to the sphere $S^3$, hence it is compact and simply connected. In contrast, both $\mathrm{Nil}^3$  and $\mathrm{\widetilde{SL}}(2,\mathbb{R})$ are non-compact, simply connected Lie groups diffeomorphic to $\mathbb{R}^3$.

\begin{lemma} With the hypotheses as above, $N$ is homeomorphic to either  $S^2$, $\R^2$,  $T^2$ or  $\R \times S^1$.	
\end{lemma}
\begin{proof} First, we consider the case $\lambda \geq 0$. 
	 An analysis of the soliton ODE in Proposition 3.4 and the proof of Theorem 1.2 in \cite{Hung1} shows that completeness and non-triviality imply $\kappa>0$. Thus, $(N, g_N)$ is isometrically homothetic to either the round $S^2$ or its quotient $\mathbb{RP}^2$. Since $(N, g_N)$ is K\"ahler, it is orientable, and hence $N\cong S^2.$

	In the case where $\lambda<0$, each Sasakian model could arise in the construction of an expanding gradient Ricci soliton \cite{DaW}. By Theorem \ref{m1.1v2}, there exists a singular orbit, corresponding to a critical value of the potential function; that is, on the singular orbit, $\nabla f=0$. Thus, each connected component of a singular orbit is a zero set for the Killing vector field $J\nabla f$. By \cite{Ko2}, it is a totally geodesic submanifold of even codimension. Thus, it is either a point or a $2$-dimensional real surface.

	If the singular orbit consists of a single point, then, since a principal orbit fibers over any singular orbit with fiber a sphere, 
each 
principal orbit must be diffeomorphic to a sphere and corresponds to the Sasakian structure with $\kappa>0$.
In this case, $N$, being the quotient of a circle fibration of $S^3$, must be $S^2$.

	If the singular orbit is a $2$-dimensional surface, then it is homogeneous with respect to the induced metric and almost complex structure. Thus, it is diffeomorphic to a one-dimensional complex homogeneous surface with constant curvature. The homogeneous space forms of dimension two of constant curvature are $S^2$, $\RP^2$, $\R^2$, $T^2$, $\R\times S^1$, and $H^2$, see, for example, \cite{Wof}. Since neither $\RP^2$ or $H^2$ admits a complex structure, the singular orbit must be homeomorphic to one of $S^2$,  $\R^2$, $T^2$, or $\R\times S^1$.\\
\end{proof}

	In summary, the Sasakian structure is a discrete quotient of a Sasakian space form in dimension three. There is a Riemannian submersion from this Sasakian manifold to a surface of constant curvature. This surface must be either the round $S^2$, 
	the flat torus, or the straight 
	 cylinder.\\ 
	
	\begin{lemma}
		\label{countdim} With the hypotheses as above, 
		then $N$ must be homeomorphic to $S^2$ or $\R^2$ and the dimension of the automorphism group is four.  
	\end{lemma}
	\begin{proof} By the proof of Theorem \ref{m1.1v2}, the Riemannian submersion from the model Sasakian to $(N, g_N, J_N)$ has each trajectory generated by the Reeb vector field mapped to a point. Thus, one immediately verifies that $(N, g_N, J_N)$ is homothetic to the leaf space with induced structure tensors as described in Subsection \ref{liftofsymmetry}. 
	Furthermore, on a $K$-contact metric structure, homogeneity implies regularity by Theorem 4 in Boothby \cite{Boo}. And by Theorem \ref{m1.1v2}, each leaf is homeomorphic to a circle and, hence, compact. Thus, we can apply Lemma \ref{liftsym} for each potential homeomorphism type of $N$. 
	
	\textit{Case $N\cong S^2$ or $\R^2$:} It is well known that the automorphism group of $(N, g_N, J_N)$ is,  
	of dimension three; thus, so is its Lie algebra. Since the first cohomology group of $N$ is trivial, Lemma \ref{liftsym} asserts that the Lie algebra of infinitesimal automorphisms on the Sasakian model is of dimension four.

	\textit{Case $N\cong T^2 \text{ or } \R \times S^1$:} On a flat torus or a cylinder, the Lie algebra of infinitesimal automorphisms is of dimension two while there is at least one infinitesimal automorphism associated with a non-trivial cohomology class. Thus, by Lemma \ref{liftsym}, the Lie algebra of infinitesimal automorphisms on the Sasakian model must be of dimension at most $1+1=2$. The result then follows. \\
	\end{proof}

	We are now ready to prove Theorem \ref{t3.1}.
	\begin{proof}[Proof of Theorem \ref{t3.1}]
	To conclude the proof, we observe that, by Lemma \ref{c3.1} and Theorem \ref{t1.1} respectively, an isometric action also preserves the almost complex structure and the potential function. Thus, it descends to a group of automorphisms on the almost contact metric structure on each regular level set of $f$ by Theorem 4.6 \cite{Hung1}. By Theorem \ref{m1.1v2}, each almost contact metric structure is a deformation of a Sasakian model. Thus, the Sasakian model is homogeneous and the Lie algebra of infinitesimal automorphisms is of dimension at least three. Then Lemma \ref{countdim}  gives the desired conclusion.  
	\end{proof}

\begin{proof}[Proof of Corollary \ref{c1.8}]
 We argue as in the proof of Theorem \ref{t3.1} to obtain $\kappa>0$ and $N\cong S^2$. Since $P$ is an $S^1$ bundle over $S^2$, the total space $P$ is compact. The classification of such bundles gives us that $P$ is either $S^1\times S^2$, or a $3$-dimensional lens space,  $L_{p, 1}$, by work of Steenrod \cite{St} ($L_{1,1}=S^3$). Since the manifold is locally irreducible, $P$ cannot be $S^1\times S^2$ by an argument similar to one in the proof of Theorem \ref{m1.1v2}. Then the result holds.
\end{proof}
  \section*{Acknowledgements}
\quad\quad This material is based upon work supported by the National Science Foundation under Grant No. DMS-1928930, while the second- and third-named authors were in residence at the MSRI in Berkeley, California, during
the Fall semester of 2024.  Part of this work was completed during the stay of the first and third-named authors at the Vietnam Institute for Advanced Study in Mathematics (VIASM) in the summer of 2025. We thank the institute for their hospitality and support. The second-named author was  partially supported by National Science Foundation grants DMS \#2506633 and \#2204324, as well as by a Simons' Foundation Travel Grant for Mathematicians \#SFI-MPS-TSM-00012804  (2025--2030).

	\bigskip
	%%%%%%%%%%%% Authors' addresses %%%%%%%%%%%%%

%
\address{{\it Ha Tuan Dung}\\
Department of Mathematics \\
	Hanoi Pedagogical University 2  \\
	Xuan Hoa, Phu Tho, Viet nam 
}
{hatuandung@hpu2.edu.vn}
\address{{\it Catherine Searle}\\
	Department of Mathematics, Statistics,\\ and Physics \\
	Wichita State University \\
	Wichita, Kansas, USA
}
{searle@math.wichita.edu}

\address{{\it Hung Tran}\\
	Department of Mathematics and Statistics\\
Texas Tech University, \\
Lubbock, TX, 79409, USA
}
{hung.tran@ttu.edu}

\end{document}